\documentclass[10pt]{amsart}

\usepackage{amssymb,latexsym,amsmath,amsfonts}
\usepackage[mathscr]{eucal}
\usepackage{mathrsfs}
\usepackage{pxfonts}
\usepackage{graphicx}
\usepackage{amscd}
\usepackage{amsthm}
\usepackage{amsxtra}
\usepackage[nointegrals]{wasysym}
\usepackage{bm}
\usepackage{calc}
\usepackage{float}
\usepackage[usenames,dvipsnames]{xcolor}
\usepackage{hyperref}
\usepackage{wrapfig}
\usepackage{verbatim}
\usepackage{xcolor}
\usepackage{graphicx}
\usepackage{cleveref}
\usepackage{orcidlink}

\theoremstyle{plain}
\newtheorem{theorem}[equation]{Theorem}
\newtheorem{proposition}[equation]{Proposition}
\newtheorem{lemma}[equation]{Lemma}
\newtheorem{corollary}[equation]{Corollary}
\newtheorem{definition}[equation]{Definition}
\theoremstyle{definition}

\numberwithin{equation}{section}

\newcommand{\sjump}{\hskip .2 cm}
\newcommand{\dbarstar}{\overline{\partial}^{\star}}
\newcommand{\psum}{\sideset{}{^{\prime}}{\sum}}
\newcommand{\dbar}{\overline \partial}

\newcommand{\ncap}{\operatorname{cap}}

\newcommand{\card}{\operatorname{card}}
\newcommand{\dom}{\operatorname{Dom}}

\newcommand{\SHU}{\operatorname{\overline{SH}}}

\begin{document}

\title[The Poincar{\'e} inequality]{On the Poincar\'{e} inequality on 
open sets in $\mathbb{R}^n$}
\author{A.-K. Gallagher\orcidlink{0000-0001-5269-2879}}
\date{\today}
\address{Gallagher Tool \& Instrument LLC, Redmond, WA 98052, USA}
\email{anne.g@gallagherti.com}

\begin{abstract}
  We show that the Poincar\'{e} inequality holds on an open set $D\subset\mathbb{R}^n$ if and only if $D$ admits a smooth, 
  bounded function whose Laplacian has a positive lower bound on $D$. Moreover, we prove that
  the existence of such a bounded, strictly subharmonic function on $D$ is equivalent to the finiteness of the strict inradius of $D$ 
  measured with respect to the Newtonian capacity.
  We also obtain a sharp upper bound, in terms of this notion of inradius, for the smallest eigenvalue of the Dirichlet--Laplacian.
\end{abstract}
\dedicatory{In memory of Peter Duren, who was a faithful friend and\\
 a wonderful companion on the historical math adventure we went on \cite{GDK081,GDK082}.}
\maketitle


\section{Introduction}

  The Poincar\'{e} inequality is said to hold on an open set $D\subset\mathbb{R}^n$ if there exists a constant $C>0$ such that
  \begin{align}\label{E:PD}
    \|f\|_{L^2(D)}\leq C\|\nabla f\|_{L^2(D)}\sjump\forall f\in\mathcal{C}^\infty_c(D),
  \end{align}
  where $\|.\|_{L^2(D)}$ is $L^2$-norm on $D$ and $\nabla$ denotes the gradient.
  
  The purpose of this paper is to characterize those open sets in $\mathbb{R}^n$, $n\geq 3$, for which the Poincar\'{e} inequality holds, in 
  terms of potential-theoretic 
  properties. One of these properties is described through the strict Newtonian inradius, $\rho_D$, of the set $D$, measured 
  with respect to the Newtonian capacity:
  \begin{align}\label{D:rhoD}
    \rho_{D}:=\sup\{R\geq 0\,|\,\forall\epsilon>0\; \exists\;x\in D:\ncap\left(\mathbb{B}(x;R)\cap D^c\right)<\epsilon\},
  \end{align}
  where $\ncap$ denotes the Newtonian capacity and $\mathbb{B}(x;R)$ the ball of radius $R$ centered at $x$; see \Cref{S:prelims} 
  for definitions.

  \begin{theorem}\label{T:MainTheorem}
    Let $D\subset\mathbb{R}^n$, $n\geq 3$,  be an open set. Then the following are equivalent.
    \begin{itemize}
       \item[(1)] The Poincar\'{e} inequality holds on $D$.
       \item[(2)] There exist a bounded function $\phi\in\mathcal{C}^{\infty}(D)$ and a positive constant $c$ such that 
       $$\triangle\phi(x)\geq c\sjump\forall\;x\in D.$$ 
       \item[(3)] $\rho_{D}<\infty$.
       \item[(4)] The maximal weak extension of the exterior derivative, acting on $(n-1)$-forms,  
       has closed range in the space of square-integrable $n$-forms on $D$.
    \end{itemize}
    If $n=2m$ for some integer $m$, then (1)-(4) are also equivalent to:
    \begin{itemize}
      \item[(5)] The maximal weak extension of the Cauchy--Riemann operator, acting on $(0,m-1)$-forms, 
      has closed range in the space of square-integrable $(0,m)$-forms on $D$.
    \end{itemize}
  \end{theorem}

  Theorem \ref{T:MainTheorem} is a natural extension of results obtained by the author in joint work with Lebl and 
  Ramachandran in \cite[Theorem 1.3]{GLR19}. In 
  \cite{GLR19}, the authors introduced 
  the notion of  strict Newtonian capacity inradius for open sets in $\mathbb{R}^2$, with the logarithmic capacity 
  in place of the Newtonian capacity, and then showed that the equivalences (1)-(3) and (5) hold in that setting. The 
  authors focused on understanding the closed range property of the Cauchy--Riemann operator,
  for planar, open sets, i.e., property (5) of  \Cref{T:MainTheorem} with $m=1$, as a prelude to the more complicated question for open 
  sets in $\mathbb{C}^n$. This was carried out by employing techniques from complex analysis. 
  As such, equivalence item (4) 
  was missed, and an extension of the results in \cite{GLR19} to $\mathbb{R}^n$ for $n\geq3$ did not appear to be within reach. 
  In the proof of (1)-(3) presented here, 
  the relevance of (4) is essentially encapsulated in Lemmata \ref{L:closedrangePD} and \ref{L:weightedestimate}.

  The equivalence of (1) and (3) in \Cref{T:MainTheorem}, with $\rho_D$ replaced by the inradius of $D$, is a
  basic fact for domains which are bounded in one direction, in particular, bounded domains, and is known to hold for domains satisfying a 
  uniform exterior cone condition, see Proposition 2.1 in \cite{Sou99} and references therein. 
  Moreover, this equivalence of (1) and (3) is related to work by Maz'ya and Shubin in \cite{MaSh05}. 
  In \cite[Theorem 1.1]{MaSh05}, they characterize those open sets $D\subset\mathbb{R}^n$, $n\geq 3$, for which the 
  Poincar\'{e} inequality holds by the finiteness of a different notion of capacity inradius, $r_D$, of $D$. We note that 
  independent of the results, Theorem 1.1 in \cite{MaSh05} and 
 (1)$\Leftrightarrow$(3) of \Cref{T:MainTheorem} herein, it is straightforward to check that $r_D<\infty$ if and only if $\rho_{D}<\infty$. 
 We also note that our methods of proof of (1)$\Leftrightarrow$(3) of \Cref{T:MainTheorem}, based on our work for $n=2$ in \cite{GLR19}, 
 noticably differ from the ones employed in \cite{MaSh05}, which is partially due to using a different notion of capacity inradius, Newtonian 
 capacity instead of Wiener capacity, and 
 classical $L^2$-techniques common in complex analysis.

  An advantage of our proof of (1)$\Rightarrow$(3) and our definition of the strict Newtonian capacity inradius is that it yields a sharp lower bound for the constant $C$ in \eqref{E:PD}, 
  i.e., a sharp 
  upper bound  for the smallest eigenvalue $\lambda_1(D)$ of the Dirichlet--Laplacian, in terms of the strict Newtonian capacity inradius, 
  $\rho_{D}$, as defined in \eqref{D:rhoD}.
  \begin{corollary}\label{C:bestupperbound}
  Let $D\subset\mathbb{R}^n$, $n\geq 3$, be an open set with $\rho_{D}<\infty$. Then
    \begin{align}\label{E:bestupperbound}
      \lambda_1(D)\leq\lambda_1(\mathbb{B})\rho_{D}^{-2}.
    \end{align}
  \end{corollary}
  It is known, and easy to see by a scaling argument, 
  that \eqref{E:bestupperbound} holds when $\rho_D$ is replaced with the inradius of $D$. We describe in detail how this standard result 
  and \Cref{C:bestupperbound} relate for those domains for which the strict Newtonian capacity inradius is attained
  within a ball of finite radius in Subsection~\ref{SS:inradii}. 
  We also note that \Cref{C:bestupperbound} is true in case of $n=2$, with the logarithmic capacity inradius 
  in place of the strict Newtonian capacity inradius. 
  A proof of this is implicitly contained in Subsection 2.3 in \cite{GLR19}.

 The proof of  (3)$\Rightarrow$(2) is done by constructing the function $\phi$ by summing up  (modifications of) 
 the potentials associated to the equilibrium measures of  
 compact, nonpolar sets which are ``nicely'' 
 distributed in the complement of the open set in consideration and whose capacities have a uniform lower bound. This kind of construction 
 has been employed in our earlier work \cite{GLR19}. However, our construction does not improve upon the lower bounds for 
 $\lambda_1$ presented in \cite{MaSh05}.

  We originally proved (2)$\Rightarrow$(1) in two steps. First, we used the $L^2$-methods for the
  twisted $\dbar$-complex, with twist factor $\phi$, applied to the $L^2$-complex of the 
  maximal weak extension of the exterior derivative on $n$-forms to show that (2)$\Rightarrow$(4) holds. Then we used standard facts for 
  Hilbert space operators  to derive (4)$\Rightarrow$(1). 
  Sean Curry pointed out to me an alternative, direct proof of (2)$\Rightarrow$(1) by Lee \cite[Lemma 7.6]{Lee06} 
  (based on Cheng--Yau \cite[pg. 345]{Cheng-Yau75}), which yields a better lower bound for $\lambda_1$. 
  Both proofs are briefly discussed in \Cref{S:(2)=>(1)}.
 
  The equivalence of (1) and (4) of \Cref{T:MainTheorem} follows from two observations. 
  Firstly, \eqref{E:PD} may be interpreted as the strong minimal extension of the exterior derivative 
  acting on functions, to have closed range in the space of square-integrable $1$-forms. 
  Secondly, this closed range property gets transferred, via the 
  Hodge-$\star$-operator, to a closed range property of the 
  maximal weak extension of the exterior derivative, acting on $(n-1)$-forms, in the space of square-integrable $n$-forms. 

  Lastly, the equivalence of (1) and (5) can be proved similarly to the one of (1) and (4). We give a somewhat different proof, 
  using that both, the strong minimal extension and the Hilbert space adjoint of 
  the weak maximal extension of 
  the Cauchy--Riemann operator acting on $(m-1)$-forms, yield Dirichlet forms for the Laplace operator on smooth, 
  compactly supported forms at the appropriate form level. 

  The paper is structured as follows. We review basic definitions and facts in \Cref{S:prelims}. Moreover, we 
  introduce this new notion of the strict Newtonian capacity inradius and derive some of its fundamental properties. 
  The proofs of the equivalences (1)$\Leftrightarrow$(4) and
  (1)$\Leftrightarrow$(5) are presented in Sections \ref{S:(1)<=>(4)} and \ref{S:(1)<=>(5)}, respectively. The remaining implications, 
  (2)$\Rightarrow$(1),  (3)$\Rightarrow$(2), and (1)$\Rightarrow$(3), are proved in Sections \ref{S:(2)=>(1)}, \ref{S:(3)=>(2)},
   and \ref{S:(1)=>(3)}, respectively.


\section{Preliminaries}\label{S:prelims}


\subsection{The Poincar\'e inequality}
  The Poincar{\'e} inequality is said to hold on an open set $D\subset\mathbb{R}^n$ if there is a constant $C>0$ such that
  \begin{align*}
     \int_{D}|f|^2\;dV\leq C\sum_{j=1}^{n}\int_{D}\left|\textstyle\frac{\partial f}{\partial x_j}\right|^2\;dV    
  \end{align*}
  holds for all $f\in\mathcal{C}^{\infty}_c(D)$, i.e., smooth functions with compact support in $D$. 
  The best constant in the Poincar\'{e} inequality for an open set $D$ is traditionally denoted by 
  (the reciprocal of the square root of) $\lambda_1(D)$, i.e., 
  $$  \lambda_1(D):=\inf\left\{\frac{\|\nabla\varphi\|^2_{L^2(D)}}{\|\varphi\|^{2}_{L^2(D)}}
  \;|\;\varphi\in H^1_0(D)\setminus\{0\} \right\}, $$
  here $H_0^1(D)$ denotes the closure of $\mathcal{C}_c^\infty(D)$ with respect to the graph norm
  $$f\mapsto\left(\|f\|_{L^2(D)}^2+\|\nabla f\|_{L^2(D)}^2 \right)^{1/2}.$$
  This nomenclature arises as $\lambda_1(D)$ may be interpreted as the smallest eigenvalue of the Dirichlet--Laplacian. In particular, 
  if $\lambda_1(D)>0$, then it is attained at some $\psi\in H_0^1(D)\setminus\{0\}$, 
  and $\psi$ is a distributional solution to the boundary value problem
  \[ 
    \begin{cases} 
      \Delta\psi+\lambda\psi=0 & \text{ on }D \\
      \psi=0 & \text{ on }bD
    \end{cases}\;,
  \]
  where $\Delta:=\sum_{j=1}^n\frac{\partial^2}{\partial x_j^2}$.
  If $D$ has sufficiently regular boundary, $\psi$ is actually the solution to the strong boundary value problem, and hence, $\lambda_1(D)$ 
  is an actual eigenvalue for the
  Dirichlet--Laplacian, see, e.g., \cite[\S 6.3 $\&$ \S 6.5]{Evans98}. This eigenvalue is of great interest as it appears in various problems 
  in mathematical physics, e.g., it is the fundamental mode of vibration for a planar membrane of given shape with fixed boundary.

  We conclude this subsection by collecting some basic properties of $\lambda_1$. One of these uses the notion of a set $E$ being polar, i.e., $E$ is contained in $\{x\in\mathbb{R}^n: \psi(x)=-\infty\}$ for some non-constant subharmonic function $\psi$ on $\mathbb{R}^n$. 
  \begin{lemma}\label{L:basicproplambda}
    Let $D\subset\mathbb{R}^n$ be open.
    \begin{itemize}
      \item[(i)]  If $D'\subset D$ is open, then $\lambda_1(D)\leq\lambda_1(D')$.
      \item[(ii)] If $D'\subset D$ is open and $D\setminus D'$ is polar, then $\lambda_1(D)=\lambda_1(D')$.
      \item[(iii)] If  $x\in\mathbb{R}^n$, then $\lambda_1(D)=\lambda_1(D+x)$.
      \item[(iv)] If $r>0$, then $r^2\lambda_1(rD)=\lambda_1(D)$.
      \item[(v)] Suppose $\{D_j\}_{j\in\mathbb{N}}\subset D$ is a sequence of increasing, 
      open sets such that $D=\bigcup_{j\in\mathbb{N}}D_j$. If $\lambda:=\inf_{j\in\mathbb{N}}\{\lambda_1(D)\}$ is positive, 
      then $\lambda_1(D)$ is positive and equals $\lambda$.
    \end{itemize}
  \end{lemma}  
  \begin{proof}
     Part (i) follows from the monotonicity property of $H_0^1(.)$, while change of variable arguments yield the homothety-translation 
     properties (iii) and (iv). 
     Part (ii) is based on the fact that  $H_0^1(D)=H_0^1(D’)$ whenever $D\setminus D’$ is polar, see, e.g., \cite[pg. 93, part (c)]{Fug99}.
     The continuity from below property, i.e., part (v), follows after making two observations. First,  (i) yields $\lambda\geq\lambda_1(D)$. 
     Second, 
     the density of $\mathcal{C}_c^\infty(D)$ in $H_0^1(D)$ and the fact that 
     $f\in\mathcal{C}^\infty_c(D)$ implies $f\in\mathcal{C}_c^{\infty}(D_j)$ for all sufficiently large $j\in\mathbb{N}$ yield
     $\lambda_1(D)\geq\lambda>0$.
  \end{proof}

 \subsection{Newtonian potentials and capacity}\label{SS:Newtonian}
 
  Let $K\subset\mathbb{R}^n$, $n\geq 3$,  be a compact set. Let $\mathcal{M}(K)$ be the set of Borel probability measures with 
  support in $K$. The function
  \begin{align}\label{E:potentialdefinition}
    p_\nu(x):=\int_{\mathbb{R}^n}|x-y|^{2-n}\;d\nu(y)
  \end{align}  
  is called the Newtonian potential associated to $\nu\in\mathcal{M}(K)$. The energy associated to such a measure $\nu$ is defined as
  $$I(\nu):=\int_{\mathbb{R}^n}\int_{\mathbb{R}^n}|x-y|^{2-n}\;d\nu(y)\;d\nu(x).$$
  It follows that $I(\nu)\in(0,\infty]$.
  If $I(\nu)=\infty$ for all $\nu\in\mathcal{M}(K)$, then the Newtonian capacity of $K$, $\ncap(K)$, is said to be zero. In this case, 
  $K$ is polar, which may be proved similarly to 
   \cite[Theorems 5.10 and 5.11]{HayKen76}. If $I(\nu)<\infty$ for some $\nu\in\mathcal{M}(K)$,
  then set 
  \begin{align*}
    \ncap(K):=\sup\left\{\textstyle\frac{1}{I(\nu)}\,|\,\nu\in\mathcal{M}(K) \right\}.
  \end{align*} 

  The notion of Newtonian capacity may be extended to Borel sets $E\subset\mathbb{R}^n$ by setting
  \begin{align}\label{E:innerreg}
    \ncap(E)=\sup\left\{\ncap(K)\,|\, K\subset E \text{ compact}\right\}.
  \end{align}  
  Moreover, outer regularity holds for the Newtonian capacity on Borel sets, i.e., for any Borel set $E\subset\mathbb{R}^n$ 
  the capacity function
  satisfies
  \begin{align}\label{E:outerreg}
    \ncap(E)=\inf\left\{\ncap(U)\,|\, E\subset U,\;\;U\subset\mathbb{R}^n\;\;\text{open}\right\},
  \end{align}
  see \cite[Theorem 2.8, Chapter II, Section 2.10]{Land72} for the Newtonian capacity satisfying inner regularity \eqref{E:innerreg} 
  and outer regularity \eqref{E:outerreg} on Borel sets.
  Some of the elemental properties of the Newtonian capacity are summarized in the following Lemma.
  \begin{lemma}\label{L:basicpropcap}
    Let $E\subset\mathbb{R}^n$ be a Borel set, $x\in\mathbb{R}^n$, and $r>0$. Then the following hold.
    \begin{itemize}
       \item[(i)]  $\ncap(E+x)=\ncap(E)$.
      \item[(ii)]  $\ncap(rE)=r^{n-2}\ncap(E)$.
      \item[(iii)] If $E'\subset E$, then $\ncap(E')\leq\ncap(E)$.
      \item[(iv)] If $\{E_i\}_{i\in\mathbb{N}}$ and $E:=\bigcup_{i=1}^\infty E_i$, then
      $$\ncap(E)\leq\sum_{i=1}^\infty\ncap(E_i).$$ 
    \end{itemize}
  \end{lemma}
  \begin{proof}
   First consider (i)-(iv) for compact sets. In that case, (i) and (ii) 
   follow from the invariance of Borel measures under translations and its behavior under 
   dilations, respectively, while proofs of (iii) and (iv) may be found in \cite[Chapter II, Section 1.5]{Land72}. 
   The general case for Borel sets may then be derived using inner regularity \eqref{E:innerreg}.
  \end{proof}

  In the case that $\ncap(K)>0$ for a compact set $K$, it can be shown that there exists a unique $\mu_K\in\mathcal{M}(K)$ such that
  \begin{align}\label{D:energy}
    I(\mu_K)=\inf\{I(\nu)\,|\,\nu\in\mathcal{M}(K) \}
  \end{align}
  see, e.g., \cite[Chapter II, Section 1.3]{Land72}\footnote{Our definition of a potential function for a given Borel probability 
  measure differs from Landkof’s by a multiplicative constant.}. 
  This measure $\mu_K$ is called the equilibrium measure associated to $K$. The potential function associated 
  to the equilibrium measure satisfies the following properties.
  \begin{lemma}\label{L:basicproppot}
    Let $K\subset\mathbb{R}^n$ be a compact set with $\ncap(K)>0$. Let $\mu$ be the equilibrium measure of $K$, and $p_{\mu}$ the   
    potential associated to 
    $\mu$. Then the following hold.
    \begin{itemize} 
      \item[(i)] $p_{\mu}$ is harmonic on $\mathbb{R}^n\setminus K$ and superharmonic on $\mathbb{R}^n$,
      \item[(ii)] $p_{\mu}\in\mathcal{C}^\infty(\mathbb{R}^n\setminus K)$,
      \item[(iii)] $p_{\mu}=I(\mu)$ holds on $K$ outside a polar set,
      \item[(iv)] $p_{\mu}(x)\leq I(\mu)$ for all $x\in\mathbb{R}^n$.
    \end{itemize}
  \end{lemma}

  \begin{proof}
    For a proof of (i), see \cite[Theorem 1.4 in Chapter I, Section 3.8]{Land72}; (ii) follows from (i). For proofs of (iii) and (iv), see 
    \cite[part (c') in Chapter II, Section 1.3]{Land72}.
  \end{proof}

  \subsection{The strict Newtonian capacity inradius -- definition and properties}\label{SS:inradii}

  In the spirit of Souplet \cite[Section 2]{Sou99}, we introduce the following inradius for open sets in $\mathbb{R}^n$.
  \begin{definition}\label{D:sNci}
    Let $D\subset\mathbb{R}^n$, $n\geq 3$, be an open set.
    The strict Newtonian capacity inradius of $D$ is given by
    \begin{align*}
      \rho_{D}=\sup\{R>0\,|\,\forall \epsilon>0\;\exists x\in D:\ncap(\mathbb{B}(x;R)\cap D^c)<\epsilon\}.
    \end{align*}
  \end{definition}
  Here, $\mathbb{B}(x;R)$ denotes the open ball of radius $R>0$ with center at $x\in\mathbb{R}^n$. 
  We abbreviate $\mathbb{B}(0;1)$ by $\mathbb{B}$.
  
  We work with an equivalent formulation for the strict Newtonian capacity inradius as described in the following lemma.
  \begin{lemma}\label{L:2nddefinradius}
    Let $D\subset\mathbb{R}^n$, $n\geq 3$. Then
    \begin{align}\label{E:2nddefinradius}
      \rho_{D}=\sup\{R>0\,|\,\forall \epsilon>0\;\exists x\in \mathbb{R}^n:\ncap(\mathbb{B}(x,R)\cap D^c)<\epsilon.\}
    \end{align}
  \end{lemma}
  \begin{proof}
    To prove \eqref{E:2nddefinradius}, denote its right hand side by $\mathfrak{R}$. It is immediate that $\rho_{D}\leq\mathfrak{R}$. 
    Now suppose that $\rho_{D}<\mathfrak{R}$. Choose 
    $R_1\in(\rho_{D},\mathfrak{R})$.  It then follows from the definitions of $\rho_D$ and $\mathfrak{R}$ that
    \begin{align*}
      &\exists\delta_0>0\;\forall x\in D: \ncap\left(\mathbb{B}(x;R_1)\cap D^c\right)\geq \delta_0,
      \;\;\text{and}\\
      &\forall\epsilon>0\;\exists x\in\mathbb{R}^n:\ncap\left(\mathbb{B}(x;R_1)\cap D^c \right)<\epsilon.
    \end{align*}
    This implies that
    \begin{align}\label{E:tempcapest}
       \forall\epsilon\in(0,\delta_0)\;\;
       \exists x\in\mathbb{R}^n\setminus D: \ncap(\mathbb{B}(x;R_1)\cap D^c)<\epsilon.
     \end{align}
     In the next step, we use the fact that $\ncap(\mathbb{B}(x,r))=cr^{n-2}$ for some $c>0$.
     Then, for any such pair $(\epsilon,x)$ in \eqref{E:tempcapest} with 
     $$(\epsilon/c)^{\frac{1}{n-2}}\leq R_1,$$
     there exists a $y\in D\cap\mathbb{B}(x;(\epsilon/c)^{\frac{1}{n-2}})$. Otherwise,  
    $\mathbb{B}(x;(\epsilon/c)^{\frac{1}{n-2}})$ would be contained in $D^c$, so that
    $$\ncap\bigl(\mathbb{B}(x;R_1)\cap D^c\bigr)\geq\ncap\bigr(\mathbb{B}(x;(\epsilon/c)^{\frac{1}{n-2}})\cap D^c\bigr)
    =\ncap\bigr(\mathbb{B}(x;(\epsilon/c)^{\frac{1}{n-2}})\bigr)
    =\epsilon.$$
    This is a  contradiction to \eqref{E:tempcapest}.
    It now follows from the choice of $y$ and \eqref{E:tempcapest} that
    \begin{align*}
      \ncap(\mathbb{B}(y;R_1-(\epsilon/c)^{\frac{1}{n-2}})\cap D^c)<\epsilon.
    \end{align*}
    Now, choose an $R_0\in(\rho_D,R_1)$. Then, after choosing $\epsilon_0>0$ such 
    that $$R_0\leq R_1-(\epsilon_0/c)^{\frac{1}{n-2}},$$ we obtain that for each $\epsilon\in(0,\epsilon_0)$ there exists a $y\in D$ such that
    $$\ncap(\mathbb{B}(y,R_0)\cap D^c)<\epsilon,$$
    and, hence, $\rho_{D}\geq R_0$. This is a contradiction to the assumption that $\rho_D<R_0$. As $R_0$ was an arbitrary value in 
    $(\rho_{D},\mathfrak{R})$, it follows that $\rho_{D}=\mathfrak{R}$.
  \end{proof}

  To understand the strict Newtonian capacity inradius for bounded domains we introduce the following notion of inradius.

  \begin{definition}\label{D:Nci}
    Let $D\subset\mathbb{R}^n$, $n\geq 3$, be an open set.
    The Newtonian capacity inradius of $D$ is given by
    \begin{align*}
        \mathfrak{r}_D=\sup\{R>0\,|\,\exists x\in \mathbb{R}^n:\ncap(\mathbb{B}(x;R)\cap D^c)=0\}.
    \end{align*}
  \end{definition}

  It follows from Definitions \ref{D:sNci} and \ref{D:Nci} 
  that $\rho_D\geq\mathfrak{r}_D$. For unbounded, open sets, the inequality may be strict. 
  For instance,
  the two inradii for $$D:=\mathbb{R}^n\setminus\bigcup_{m\in\mathbb{Z}^n\setminus\{0\}}\mathbb{B}(m,|m|^{-1})$$
  are $\rho_{D}=\infty$ and $\mathfrak{r}_D=\sqrt{n}/2$. However, for bounded, open sets, these two notions agree.

  \begin{lemma}\label{L:pol=cap}
    Let $D\subset\mathbb{R}^n$, $n\geq 3$,  be an open and bounded set. Then $\rho_D=\mathfrak{r}_{D}$.
  \end{lemma}

  \begin{proof}
    Suppose $\mathfrak{r}_D<\rho_D$. Let $R_1\in(\mathfrak{r}_D,\rho_D)$. Then
    \begin{align*}
      \ncap\left(\mathbb{B}(x,R_1)\cap D^c \right)>0\;\;\sjump\;\;\forall\;x\in \mathbb{R}^n,
    \end{align*}
    and for all $\epsilon>0$ there exists an $x\in \mathbb{R}^n$ such that
    \begin{align*}
      \ncap\left(\mathbb{B}(x,R_1)\cap D^c \right)<\epsilon. 
    \end{align*}
    The latter implies that there exists a sequence $\{x_j\}_{j\in\mathbb{N}}\subset \mathbb{R}^n$ such that
    $$\ncap\left(\mathbb{B}(x_j,R_1)\cap D^c\right)<\frac{1}{j}.$$
    Since $D$ is bounded, it follows that $\{x_j\}_{j\in\mathbb{N}}$ is a bounded set and, hence, has a convergent subsequence 
    $\{x_{j_k}\}_{k\in\mathbb{N}}$. Let $\hat{x}$ be the limit of $\{x_{j_k}\}_{k\in\mathbb{N}}$. 
    After choosing 
   $R_0\in(\mathfrak{r}_D,R_1)$ and $\delta\in(0, R_1-R_0)$, we may choose a $k_0\in\mathbb{N}$ such that
   \begin{align*}
     |\hat{x}-x_{j_k}|<\delta\;\;\forall\;\;k\geq k_0. 
   \end{align*}
   This implies that $\mathbb{B}(\hat{x}, R_0)\subset\mathbb{B}(x_{j_k},R_1)$ for all $k\geq k_0$, and hence
   \begin{align*}
     \ncap\left(\mathbb{B}(\hat{x},R_0)\cap D^c\right)\leq\ncap\left(\mathbb{B}(x_{j_k},R_1)\cap D^c \right)\leq \frac{1}{j_k} 
     \;\;\sjump\forall\;\;k\in\mathbb{N}. 
   \end{align*}
   Taking the limit as $k\to\infty$ yields
   $$\ncap \left(\mathbb{B}(\hat{x},R_0)\cap D^c\right)=0$$
   which is a contradiction to $R_0>\mathfrak{r}_D$. 
 \end{proof}
 
  It is now easy to see that $\rho_D$ and $\mathfrak{r}_D$ are equal whenever $D\subset\mathbb{R}^n$ is such that
  $\rho_D$ is attained within a ball of finite radius. For sake of brevity, we write $D_R$ for $D\cap\mathbb{B}(0;R)$ in the following.
  
  \begin{corollary}
    Let $D\subset\mathbb{R}^n$ be an open set such that $\rho_D=\rho_{D_R}$ for some $R>0$. Then
    $\rho_D=\mathfrak{r}_D$.
  \end{corollary}
  
  \begin{proof}
    It follows from \Cref{L:pol=cap} that $\rho_D=\mathfrak{r}_{D_R}$. Monotonicity of the 
    Newtonian inradius yields $\mathfrak{r}_{D_R}\leq\mathfrak{r}_D$, while $\mathfrak{r}_D\leq\rho_D$ holds by definition. 
    Hence, the  claim follows.
  \end{proof}
  
  Another observation on the strict Newtonian inradius is its invariance under polar sets.
  \begin{lemma}\label{L:rhopolarinv}
    Let $D\subset D’\subset\mathbb{R}^n$, $n\geq 3$, be open sets such that $D’\setminus D$ is polar. Then $\rho_{D}=\rho_{D’}$.
  \end{lemma}

  \begin{proof}
    First, note that $\rho_{D}\leq\rho_{D’}$ follows from the monotonicity of the strict Newtonian capacity. 
    Next, subadditivity of the Newtonian capacity yields
    \begin{align*}
      \ncap\left(\mathbb{B}(x;R)\cap D^c \right)&\leq
      \ncap\left(\mathbb{B}(x;R)\cap D’\setminus D \right)+\ncap\left(\mathbb{B}(x;R)\cap (D’)^c \right)\\
      &=\ncap\left(\mathbb{B}(x;R)\cap (D’)^c \right),
    \end{align*}
    where the last step follows from $D’\setminus D$ being polar. Hence, $\rho_{D’}\leq \rho_D$, which yields $\rho_D=\rho_D’$.
  \end{proof}

  This ties in with the fact that the Poincar{\'e} inequality is invariant under removal of polar sets. 
  That is, \eqref{E:PD} holds for an open set $D$ if and only if it holds for 
  an open set $D’$, with the same constant, whenever $D$ and $D’$ differ by a polar set. 
  This is due to the $L^2$-Sobolev-$1$-space being invariant under 
  removal of polar sets, see, e.g., \cite[pg. 93]{Fug99} and references therein.

  We conclude this subsection by showing that any  open set $D\subset\mathbb{R}^n$, with $\rho_D=\rho_{D_R}$ for some $R>0$, 
  may be associated to another open set 
  $\widehat{D}$ such that $\rho_D$ equals the inradius, $\mathcal{R}_{\widehat{D}}$, of $\widehat{D}$, i.e.,
  $$\mathcal{R}_{\widehat{D}}:=\sup\{R>0\;|\;\exists x\in\widehat{D}: \mathbb{B}(x;R)\subset \widehat{D}\}. $$
  For that, we denote  the family of subharmonic functions on $D$ which are bounded from the above on $D$ by $\SHU(D)$. We then
   associate the following open set $\widehat{D}$ to $D$:
  \begin{align*}
    \widehat{D}:=\bigcup\left\{U\subset\mathbb{R}^n \;\text{open}\;
    |\;D\subset U, \forall\varphi\in\SHU(D)\;\exists\varphi_U\in\SHU(D\cup U): 
    (\varphi_U)_{|_{D}}=\varphi\right\}.
  \end{align*}
  We first collect a few basic properties satisfied by $\widehat{D}$.
  \begin{lemma}\label{L:Dhatprops}
    Let $D\subset\mathbb{R}^n$, $n\geq 3$, be open. Then
    \begin{itemize}
      \item[(i)] $\widehat{D}$ is open.
      \item[(ii)] $\widehat{D}\setminus D$ is polar.
      \item[(iii)] For all $\varphi\in\SHU(D)$ there exists a unique $\widehat{\varphi}\in\SHU(\widehat{D})$ 
      such that $\widehat{\varphi}_{|_D}=\varphi$.
      \item[(iv)] $\widehat{D}$ is bounded whenever $D$ is.
    \end{itemize}
  \end{lemma}

  \begin{proof}
    Part (i) follows directly from the definition of $\widehat{D}$.
 
    To see part (ii), note first that if $\widehat{D}\setminus D$ was nonpolar, then there would exist an open set $U\subset\widehat{D}$, 
    which contains $D$, such that 
    $U\setminus D$ is nonpolar and any function in $\SHU(D)$ extends to a function in $\SHU(D\cup U)$. This implies that the potential 
    function associated to 
    the equilibrium measure of a nonpolar, compact 
    subset of $U\setminus D$ extends to a positive function in $\SHU(\mathbb{R}^n)$ whose values approach zero at infinity. 
    This is impossible due to the maximum principle for subharmonic functions; for 
    more details on this argument, see \cite[pg. 239]{HayKen76}\footnote{Note that potential functions in \cite{HayKen76} differ 
    by a minus sign from our definition \eqref{E:potentialdefinition}.}.
 
    Part (iii) follows from (ii) by Theorem 5.18 in \cite{HayKen76}\footnote{This theorem is attributed 
    to Brelot \cite{Brelot34} in \cite{HayKen76}. We refer to the latter because of lack of accessibility to the former.}. 
    Part (iv) also follows from (ii), since, if $\widehat{D}$ was 
    unbounded, $\widehat{D}\setminus D$ would contain nonpolar sets.
  \end{proof}

  The point of the set $\widehat{D}$ is that it is obtained from $D$ by filling in polar holes, i.e., it is obtained by unifying $D$ with those  
  connected components of $bD$ which are polar and disconnected from the complement of $\overline{D}$.

  \begin{lemma}
    Let $D\subset\mathbb{R}^n$ be an open set such that $\rho_{D}=\rho_{D_R}$ for some $R>0$. Then 
    $\rho_D=\rho_{\widehat{D}}=\mathcal{R}_{\widehat{D}}$.
  \end{lemma}

  \begin{proof}
     We first prove this lemma in the case that $D$ is bounded. For that, we note that $\widehat{D}\setminus D$ is 
     polar by \Cref{L:Dhatprops}.
     Thus, it follows from \Cref{L:rhopolarinv} that $\rho_D=\rho_{\widehat{D}}$. By definition of the inradii, we also have
     $\mathcal{R}_{\widehat{D}}\leq\rho_{\widehat{D}}$.  Suppose $\mathcal{R}_{\widehat{D}}<\rho_{\widehat{D}}$. 
     Let $R_0, R_1\in\mathbb{R}$ such that
    \begin{align*}
      \mathcal{R}_{\widehat{D}}<R_0<R_1<\rho_{\widehat{D}}.
     \end{align*}
    Then there exists an $x_0\in\mathbb{R}^n$ such that
    \begin{align}\label{E:tempcap0}
      \ncap\left(\mathbb{B}(x_0;R_1)\cap(\widehat{D})^c \right)&=0,\;\;\text{and}\\
      \mathbb{B}(x_0;R_0)\cap(\widehat{D})^c&\neq\varnothing.\notag
    \end{align}
    It follows that
   $$\mathbb{B}(x_0;R_0)\cap(\widehat{D})^c= \mathbb{B}(x_0;R_0)\cap b\widehat{D},$$ 
   since otherwise the intersection of $\mathbb{B}(x_0;R_0)$ with the interior of $(\widehat{D})^c$ would yield a nonempty and open, 
   hence, nonpolar set which is a contradiction to \eqref{E:tempcap0}.
 
   It suffices to show that there exists an open set $U$ containing $\mathbb{B}(x_0;R_0)\cap b\widehat{D}$ 
   such that $U\subset\widehat{D}\cup b\widehat{D}$. To wit, if such an open set $U$ exists, then $U\subset\widehat{D}$ 
   which implies that $\mathbb{B}(x_0;R_0)\cap b\widehat{D}=\varnothing$. 
   Hence, $\mathcal{R}_{\widehat{D}}\geq R_0$ would hold which is a contradiction. Since $R_0$ was 
   chosen arbitrarily in $(\mathcal{R}_{\widehat{D}},\rho_{\widehat{D}})$, 
   the proof of $\mathcal{R}_{\widehat{D}}=\rho_{\widehat{D}}$ for $D$ bounded would then be completed.
 
   If no such open set $U$ exists, then we could choose an open set $V$ such that
   $$\mathbb{B}(x_0;R_0)\cap b\widehat{D}\subset V \subset\mathbb{B}(x_0;R_1)$$ 
   and the intersection of $V$ with the interior of $(\widehat{D})^c$ is nonempty and open. Hence,
   \begin{align*}
     V\cap\left(\overline{\widehat{D}}\right)^c\subset\mathbb{B}(x_0;R_1)\cap(\widehat{D})^c.
   \end{align*}
   This is a contradiction as the set on the left hand side is nonpolar while the set on the right hand side is polar.
   
   Now suppose that $D\subset\mathbb{R}^n$ is such that $\rho_D=\rho_{D_R}$ for some $R>0$. Then
   $\rho_{D_R}=\mathcal{R}_{\widehat{D}_R}$ since $D_R$ is bounded. It follows from (ii) of \Cref{L:Dhatprops} and 
   Theorem 5.18 in \cite{HayKen76} that $\widehat{D_R}\subset\widehat{D}_R$. 
   Thus, we have so far
   $$\rho_D=\rho_{D_R}=\mathcal{R}_{\widehat{D_R}}\leq\mathcal{R}_{\widehat{D}_R}.$$
   But we also have by monotonicity, definition, and polarness of $\widehat{D}\setminus D$, that
   $$\mathcal{R}_{\widehat{D}_R}\leq\mathcal{R}_{\widehat{D}}\leq\rho_{\widehat{D}}=\rho_D,$$
   which concludes that proof.
  \end{proof}

  \subsection{The exterior derivative -- extensions and adjoints}\label{SS:extD}

  Denote by $(x_1,\dots,x_n)$, with $x_j\in\mathbb{R}$, the Euclidean coordinates of $\mathbb{R}^n$, and by $dx_j$  
  the differential of the coordinate $x_j$, $j\in\{1,\dots,n\}$. For $k\in\{1,\dots,n\}$ and a multi-index $I=(i_1,\dots,i_k)$, 
  $i_j\in\{1,\dots,n\}$, of length $k$, write $dx^I$ in place of $dx^{i_1}\wedge\dots\wedge dx^{i_k}$. 
  Recall that $\{dx^I: I=(i_1,\dots,i_k), i_j<i_{j+1}\}$ forms a basis of the space of differential $k$-forms. 
  That is,
  if $u$ is a differential $k$-form, then there exist unique functions $u_I$   such that
  \begin{align}\label{E:standardkformrepresentation}
    u=\psum_{|I|=k}u_I dx^I,
  \end{align}  
  where the prime indicates that the sum is taken only over increasing multi-indices.

  Let $D\subset\mathbb{R}^n$ be an open set. For $k\in\{0,1,\dots,n\}$, let $\Omega^k(D)$ be the space of $k$-forms whose 
  coefficient functions are smooth in $D$. That is, $\Omega^{0}(D)=\mathcal{C}^\infty(D)$, and
  $u\in\Omega^k(D)$ iff the coefficient functions $u_I$ in \eqref{E:standardkformrepresentation} belong to $\mathcal{C}^\infty(D)$. Similarly, 
  set $\Omega^k_c(D)$ to be the space of  differential $k$-forms whose 
  coefficient functions are smooth and have compact support in $D$.

  Denote by $L^2(D)$ the space of square-integrable functions, write $\|.\|_{L^2(D)}$  and $(.,.)_{L^2(D)}$ for the 
  norm and inner 
  product on $L^2(D)$, respectively. Further, write $L^2_k(D)$ for the space of differential $k$-forms with square-integrable coefficients; drop 
  the index when $k=0$. If
  $u\in L^2_k(D)$, then the $L^{2}_{k}(D)$-norm of $u$, represented as in \eqref{E:standardkformrepresentation}, is given by
  $$\left\|u\right\|_{L^2_k(D)}^2=\psum_{|I|=k}\left\|u_I\right\|^2_{L^2(D)}.$$ 
  $L^2_k(D)$ inherits the inner product from $L^2(D)$. 
  The exterior derivative $d_k$ is initially defined on $\Omega^k(D)$ by
  \begin{align*}
    d_k u=\psum_{|I|=k}\sum_{j=1}^n\frac{\partial u_I}{\partial x_j}dx^j\wedge dx^I
  \end{align*}
  for $u\in\Omega^k(D)$ as in \eqref{E:standardkformrepresentation}.
 
  To define the weak maximal extension of the exterior derivative, first extend $d_k$  to act on 
  $L^2_k(D)$ in the sense of distributions; denote the extension also by $d_k$. Then restrict its domain to the subspace
  $$\dom(d_k)=\left\{u\in L^2_k(D)\,|\, d_k u\in L^2_{k+1}(D) \right\}.$$ Then $d_k$ is a densely defined, closed operator on $L^2_k(D)$.
  We denote by $d_k^\star$ its Hilbert space adjoint.
  Its domain is given by
   \begin{align*}
     \dom(d_k^{\star})=\left\{v\in L^2_{k+1}(D)\,|\,\exists\;C>0:
     |(d_k u,v)_{L^2_{k+1}(D)}|\leq C\|u\|_{L^2_k(D)}\;
     \;\forall u\in\dom(d_k) \right\}. 
  \end{align*}
%
 
  We also consider the strong minimal extension, $d_{k,c}$, of $d_k$. This extension is obtained by 
  first restricting $d_k$ to $\Omega^k_c(D)$, and then taking the closure of the resulting operator  with respect to the graph norm
  $$u\mapsto \left(\|u\|_{L^2_k(D)}^2+\|d_k u\|_{L^2_{k+1}(D)}^2\right)^{1/2}.$$ Note that $\dom(d_{0,c})=H_0^1(D)$. Hence
  the Poincar\'{e} inequality \eqref{E:PD} may be reformulated in terms of $d_{0,c}$ as
  \begin{align}\label{E:dcompactclosedrange}
    \|f\|_{L^2(D)}\leq C\|d_{0,c} f\|_{L^2_1(D)}\;\;\sjump\forall f\in\dom(d_{0,c}).
  \end{align}

  \subsection{The Cauchy--Riemann operator as a densely defined $L^{2}$-operator}
  Write $z_{j}=x_{2j-1}+ix_{2j}$ for $(x_{1},\dots,x_{2n})$ the Euclidean coordinates of $\mathbb{R}^{2n}$, $dz^{j}$ 
  for the differential of 
  the coordinate $z_{j}$ and $d\bar{z}^{j}$
  for its conjugate . A $(0,k)$-form $d\bar{z}^{j_{1}}\wedge\dots\wedge d\bar{z}^{j_{k}}$ may be abreviated as $d\bar{z}^J$ 
  for the multi-index $J=(j_1,\dots,j_k)$.  In direct analogy to \Cref{SS:extD}, we 
  may represent any differential $(0,k)$-form $u$
  as
  \begin{align}\label{E:standard0kformrepresentation}
    u=\psum_{|J|=k}u_Jd\bar{z}^J,
  \end{align}
  for uniquely determined functions $u_J$. Also in complete analogy to \Cref{SS:extD}, for $D\subset\mathbb{C}^m$ open, 
  we may define $\Omega^{0,k}(D)$, $\Omega^{0,k}_c(D)$ and $L^2_{0,k}(D)$ as the spaces of $(0,k)$-forms with smooth, 
  smooth and compactly supported, and $L^2(D)$-integrable coefficients, respectively. The Cauchy--Riemann operator, defined as
  $$\dbar_{k}u=\psum_{|J|=k}\sum_{j=1}^m\frac{\partial u_J}{\partial \bar{z}_j}d\bar{z}^j\wedge d\bar{z}^J,$$
  for $u\in\Omega^{0,k}(D)$ represented as in \eqref{E:standard0kformrepresentation}. Similarly to \Cref{SS:extD}, we then consider the  
  weak maximal extension of $\dbar_k$, still calling it $\dbar_k$.
  This construction yet again yields a densely defined, closed operator on $L^2_{0,k}(D)$, 
  with a Hilbert space adjoint denoted by $\dbarstar_{k}$.

  \section{Proof of  \texorpdfstring{(1)$\Leftrightarrow$(4)}{(1)<=>(4)}}\label{S:(1)<=>(4)}
 
 
 The proof of (1)$\Leftrightarrow$(4) utilizes
 standard characterizations of the closed range property of linear, closed, densely defined Hilbert space operators. 
 If $$T:\dom(T)\subset\mathcal{H}_1\longrightarrow~ \mathcal{H}_2$$ is such an operator for some 
 Hilbert spaces $\{\mathcal{H}_j\}_{j=1,2}$, 
  then $T$ has closed range in $\mathcal{H}_2$ if and only if there exists a $C>0$ such that
  \begin{align}\label{E:closedrange}
     \|u\|_{\mathcal{H}_1}\leq C\|T u\|_{\mathcal{H}_2}\;\;\sjump\forall\;\;u\in\dom(T)\cap(\ker(T))^\perp,
  \end{align}    
  or, if and only if its Hilbert space adjoint, $T^\star$, has closed range in $\mathcal{H}_1$. 
  The closed range property of $T^\star$ may also be expressed as an estimate, that is, $T^\star$ has 
  closed range if and only if there exists a $C>0$ such that
  \begin{align}\label{E:closedrangeT*}
     \|v\|_{\mathcal{H}_2}\leq C\|T^\star v\|_{\mathcal{H}_1}\;\;\sjump\forall\;\;v\in\dom(T^\star)\cap(\ker(T^\star))^\perp.
   \end{align}    
   Note that the 
  best constants in the estimates \eqref{E:closedrange} and \eqref{E:closedrangeT*} for $T^\star$ are equal. See
  e.g., \cite[Theorem 1.1.1]{Hormander65} for proofs of these facts.

  Henceforth, if $d_k$ has closed range in $L^2_{k+1}(D)$, we write $\mathfrak{C}_k(D)$ for the best 
  such constant, otherwise $\mathfrak{C}_k(D)=\infty$. That is,
  \begin{align*}
     \mathfrak{C}_k(D)=\inf\Bigl\{C\in\mathbb{R}^+\cup\{\infty\}\,|\,\|u\|_{L^2_k(D)}
     \leq C\|d_k u\|_{L^2_{k+1}(D)}
     \;\forall
     u\in\dom(d_k)\cap(\ker(d_k))^\perp \Bigr\}.
  \end{align*}
 
  It follows from the observation leading to \eqref{E:dcompactclosedrange} that we may 
  reformulate the equivalence  (1)$\Leftrightarrow$(4) as follows.
  \begin{lemma}\label{L:closedrangePD}
    Let $D\subset\mathbb{R}^n$, $n\geq 1$, be an open set.  Then the following are equivalent.
    \begin{itemize}
       \item[(a)]  There exists a constant $C>0$ such that
       \begin{align}\label{E:closedrangedcompact}
          \|f\|_{L^2(D)}\leq C\|d_{0,c}f\|_{L^2_1(D)}\;\;\forall f\in \dom(d_{0,c}).
       \end{align}  
       \item[(b)]  There exists a constant $C>0$ such that
         $$\|u\|_{L^2_{n-1}(D)}\leq C\|d_{n-1}u\|_{L^2_n(D)}\;\;\forall u\in\dom(d_{n-1})\cap(\ker(d_{n-1}))^{\perp}.$$
     \end{itemize}
     The best constants in (a) and (b) are finite iff $\lambda_1(D)>0$; if they are finite, then they equal  $(\lambda_1(D))^{-1/2}$.
  \end{lemma}
  \begin{proof}
    First observe that $\ker{d_{0,c}}=\{0\}$ as $H_0^1(D)$ 
    cannot contain any non-trivial functions which are constant on the connected components of $D$.  
    The latter can be seen by a proof analogous to the one given in \cite[Lemma 2.10]{GLR19}.
    This observation implies that (a) is equivalent to $d_{0,c}$ having closed range in $L^2_{1}(D)$, with the 
    best constant being equal to $(\lambda_1(D))^{-1/2}$.

   The remainder of the proof essentially follows from three basic facts:
   \begin{itemize}
      \item[(i)] the Hodge star operator~$\star$ is an isometry 
        between $L^2_k(D)$ and $L^2_{n-k}(D)$ and a bijection between $\Omega^k_c(D)$ 
        and $\Omega^{n-k}_c(D)$ for $0\leq k\leq n$, 
      \item[(ii)] the formal adjoint, $\vartheta_{n-1}$, of the differential operator $d_{n-1}$ is equal to $-\star d_{0}\star$,
      \item[(iii)] the strong minimal extension of $\vartheta_{n-1}$ is $d_{n-1}^\star$, i.e.,
     $d_{n-1}^\star=-\star d_{0,c}\star$.
   \end{itemize}
   Facts (i) and (ii) follow straightforwardly from the definition of the Hodge star operator. For (iii), note first that
   $d_{n-1}^\star$ and $\vartheta_{n-1}$ are equal on $\Omega^n_c(D)$. Furthermore, 
   $\Omega_c^n(D)$ is dense in $\dom(d^\star_{n-1})$ with respect to the graph norm given by
   $$v\mapsto\left(\|v\|^2_{L_n^2(D)}+\|d_{n-1}^\star v\|_{L^2_{n-1}(D)}^2 \right)^{1/2}.$$
   This may be shown analogously to the proof for (i) of Proposition~2.3 in \cite{Straube10}; just 
   replace $\overline{\partial}$ by $d_{n-1}$, hence $\overline{\partial}^\star$ by $d_{n-1}^\star$, therein, and note that the boundedness 
   assumption in \cite{Straube10} is irrelevant for the proof. Thus (iii) holds. 
 
    It follows from (i) and (iii) that
   $d_{0,c}=(-1)^n\star d_{n-1}^\star\star$, and
   \begin{align*}
     \star\dom(d_{0,c})=\dom(d_{n-1}^\star),\;\star\ker(d_{0,c})=\ker(d_{n-1}^\star),\\
     \star\dom(d_{n-1}^\star)=\dom(d_{0,c}),\;\star\ker(d_{n-1}^\star)=\ker(d_{0,c})
   \end{align*}
   Therefore, \eqref{E:closedrangedcompact}, with $v:=\star f$, is equivalent to
   $$\|v\|_{L^2_n(D)}\leq C\|d_{n-1}^\star v\|_{L^2_{n-1}(D)}\sjump\forall v\in\dom(d_{n-1}^*)\cap\ker(d_{n-1}^*)^\perp$$ 
   with the best constant equal to $(\lambda_1(D))^{-1/2}$. By the remark at the beginning of this section, this is equivalent to $d_{n-1}$  
   having closed range
   in $L^2_n(D)$ with $\mathfrak{C}_{n-1}(D)=(\lambda_1(D))^{-1/2}$.
  \end{proof}
 



  \section{Proof of \texorpdfstring{(2)$\Rightarrow$(1)}{(2)=>(1)}}\label{S:(2)=>(1)}


  We briefly elaborate on two different proofs of the implication (2)$\Rightarrow$(1). The first one is 
  based on the Kohn--Morrey--H\"ormander 
  formula for a twisted $\dbar$-complex in the sense of Ohsawa--Takegoshi, see Section 2.6 in \cite{Straube10} 
  and references therein ; the second one is due to Lee \cite[Lemma 7.6]{Lee06} and based on 
  a result of 
  Cheng--Yau \cite[pg. 345]{Cheng-Yau75}.
 
  \begin{lemma}\label{L:weightedestimate}
    Let $D\subset\mathbb{R}^n$ be an open set. Let $\phi\in\mathcal{C}^2(D)$ be bounded from above by $M\in\mathbb{R}$.
    Then
    \begin{align}\label{E:weightedestimate}
      \int_D\Delta\phi\cdot (\star w)^2 e^{\phi-M}dV\leq\|d_{n-1}^\star w\|_{L^2_{n-1}(D)}^2\;\;\sjump\;\;\forall w\in\Omega^n_c(D).
    \end{align}
  \end{lemma}

  \begin{proof}
    We first note that Proposition 2.4 in \cite{Straube10} holds for the $d$-complex on $\Omega^{n}_c(D)$ 
    with the twist factor $a=1-e^{\phi-M}$ and weight $\varphi=0$. Furthermore, 
    we note that we may drop the boundedness and smoothness assumptions on the domain
    in Proposition 2.4 in \cite{Straube10} since we only consider compactly supported forms. 
    Inequality \eqref{E:weightedestimate} may then be derived analogously to inequality (2.48) in \cite[Lemma 2.6]{Straube10}. 
    Unlike in \cite[Lemma 2.6]{Straube10}, no geometric boundary assumptions are needed as, again, 
    we only consider compactly supported forms.
  \end{proof}

  \begin{corollary}\label{C:Hormanderlambda}
    Let $D\subset\mathbb{R}^n$ be an open set. Suppose $D$ admits a bounded function $\phi\in\mathcal{C}^2(D)$ such that
    $\Delta\phi\geq c$ on $D$ for some constant $c>0$. Then 
    \begin{align*}
     \|\omega\|_{L^2(D)}\leq C\|d_{0,c}\omega\|_{L^2_1(D)}\;\;\sjump\forall\;\;\omega\in\mathcal{C}_c^\infty(D)
    \end{align*} 
    holds for $C=\sqrt{e^{M-m}/c}$ where $m$ and $M$ are a lower and an upper bound of $\phi$ on $D$, respectively. 
    That is, the Poincar{\'e} inequality holds and $\lambda_1(D)\geq e^{m-M}c$.
  \end{corollary}

  \begin{proof}
    Let $\omega\in\mathcal{C}_c^{\infty}(D)$ be given, and set $w=\star \omega$. Then $w\in\Omega^n_c(D)$, so 
    that \eqref{E:weightedestimate} holds for $w$. In particular,
    $$ \int_D(\star w)^2 dV\leq C^2\|d_{n-1}^\star w\|_{L^2_{n-1}(D)}^2\;\;\sjump\;\;\forall w\in\Omega^n_c(D)$$
    holds with $C=\sqrt{e^{M-m}/c}$. It follows from the Hodge star operator being an isometry 
    between $L^2(D)_*$ and $L^2_{n-*}(D)$ and 
    the identity $d_{n-1}^\star=-\star d_{0,c}\star$ that
    \begin{align*}
      \|\omega\|_{L^2(D)}=\|\star w\|_{L^2(D)}\leq C\|d_{n-1}^\star w\|_{L^2_{n-1}(D)}=C\|d_{0,c}\omega\|_{L^2_1(D)}.
    \end{align*}  
    However, this means that the Poincar\'e inequality holds with $\lambda_1(D)\geq ce^{m-M}$.
  \end{proof}

  The following lemma, in a more general setting, is due to Lee \cite[Lemma 7.6]{Lee06}\footnote{Lee uses 
  $\Delta=-\sum_{j=1}^n\frac{\partial^2}{\partial x_j^2}$.}
  \begin{lemma}\label{L:Lee}
    Let $D\subset\mathbb{R}^n$ be an open set. Suppose there exist a positive 
    function $\phi\in\mathcal{C}^2(D)$ and a constant $\lambda>0$ such that $-\Delta\phi/\phi\geq\lambda$ on $D$. 
    Then the Poincar{\'e} inequality holds and $\lambda_1(D)\geq\lambda$.
  \end{lemma}
  \begin{proof}
    Lee first proves that 
    $$0=\int_D d_0^{\star}\left(\omega^2\phi^{-1}d_0\phi\right)\;dV$$
    holds for all $\omega\in\mathcal{C}^\infty_c(D),$
    which then is used to derive the identity
    $$\int_D \left(\omega^2\phi^{-1}\Delta\phi+|\nabla \omega|^2\right)\;dV=\left\|\phi d_0(\phi^{-1}\omega)
     \right\|_{L^2_{1}(D)}^2.$$
    The non-negativity of the right hand side and the hypothesis on $\phi$ then imply that the Poincar{\'e} inequality holds with 
    $\lambda_1(D)\geq \lambda$.
  \end{proof}
  Note that, with the hypotheses of \Cref{C:Hormanderlambda}, \Cref{L:Lee} yields $\lambda_1(D)\geq c/(\widetilde{M}-m)$ for any 
  $\widetilde{M}>M$. That is, \Cref{L:Lee} yields a better lower bound for $\lambda_1(D)$ than \Cref{C:Hormanderlambda}.


  \section{Proof of  \texorpdfstring{(3)$\Rightarrow$(2)}{(3)=>(2)}}\label{S:(3)=>(2)}


  \begin{proof}[Proof of (3)$\Rightarrow$(2) of \Cref{T:MainTheorem}]
    Suppose $\rho_{D}$ is finite. Let $M>\rho_{D}$. By \Cref{L:2nddefinradius} there exists a $\delta>0$ such that for all $x\in\mathbb{R}^n$
    \begin{align*}
      \ncap\left(\mathbb{B}(x,M)\cap D^{c} \right)\geq 2\delta.
    \end{align*}
    This lets us choose a sequence of well-spread out, compact sets whose Newtonian capacity is greater than or equal to $\delta$.
    In fact, define $\mathcal{Q}(Nm,L)$, $m\in\mathbb{Z}^n$, $N, L\in\mathbb{N}$, to be the closed $n$-cube 
    with center $Nm$ and side length $L$.
    Then, for any $m\in\mathbb{Z}^n$, we may choose a compact set $K_m\subset D^{c}$ such that
    \begin{align*}
      K_m\Subset \mathcal{Q}(2Mm,2M) \text{ and }\ncap(K_{m})\geq\delta.
    \end{align*}
    For each $m\in\mathbb{Z}^n$, let $\mu_{m}$ be the equilibrium measure of $K_m$. Let $p_{m}$ be the potential function associated to
   $\mu_{m}$. By \Cref{L:basicproppot}, $p_m$ is in $\mathcal{C}^{\infty}(D)$, harmonic on $D$, and 
   \begin{align}\label{E:basicestp}
     0< p_m\leq 1/\ncap(K_m)\leq1/\delta.
   \end{align}  
   In the following, we shall show that 
   \begin{align}\label{E:1stdefphi}
     \sum_{m\in\mathbb{Z}^{n}}e^{-\frac{4}{\delta p_m}}
   \end{align} 
   is a well-defined, bounded, smooth, subharmonic function on $D$ whose Laplacian has a uniform 
   positive lower bound on $D$. For sake of brevity, we will
   write ``$A\lesssim B $" for non-negative quantities $A$ and $B$ to mean that there exists some 
   constant $c>0$ depending on $n$ and $M$ such that $A\leq c B$ holds.
  
   To show that \eqref{E:1stdefphi} yields a well-defined function at any given $x\in D$, we order $\mathbb{Z}^n$, 
   roughly speaking, by distance to $x$. For that, we equip $\mathbb{Z}^n$ with the
   maximum norm $\|.\|_{\max}$, i.e.,
   $$\|m\|_{\max}:=\max_{j\in\{1,\ldots,n\}}
   \left\{|m_j|\right\}\;\;\text{for}\;\;m=(m_1,\ldots,m_n)\in\mathbb{Z}^n.$$ 
   Then, for each $x\in D$, we choose a $m(x)\in\mathbb{Z}^n$ such that $x\in\mathcal{Q}(2Mm(x),2M)$, and define
   \begin{align*}
     A_\lambda(x):=\Bigl\{m\in\mathbb{Z}^n\,|\,\lambda=\|m(x)-m\|_{\max}\Bigr\}\;\;\text{for}\;\;\lambda\in\mathbb{N}_0.
  \end{align*}
  Now, consider 
  $$\phi(x):=\sum_{\lambda\in\mathbb{N}_0}\sum_{m\in A_{\lambda}(x)}e^{-\frac{4}{\delta p_m(x)}}\;\;\text{for}\;\;x\in D.$$
  To see that $\phi(x)$ exists for all $x\in D$, we first determine the cardinality of $A_\lambda(x)$. Clearly, $A_0(x)=\{m(x)\}$, and
  \begin{align*}
    \card\Bigl(\bigcup_{k=0}^\ell A_k(x)\Bigr)=\card\left(\mathbb{Z}^n\cap[-\lambda,\lambda]^n\right)
    =(2\lambda+1)^n\sjump\;\;\forall\ell\in\mathbb{N}_0.
  \end{align*}
  Since
  $$A_\lambda(x)=\bigcup_{k=0}^\lambda A_{k}(x)\setminus \bigcup_{k=0}^{\lambda-1} A_{k}(x),$$
  it follows that $\card(A_{\lambda}(x))=(2\lambda +1)^n-(2\lambda-1)^n$. In particular, there exists a constant $c(n)>0$ such that
  $\card(A_{\lambda}(x))\leq c(n)\lambda^{n-1}$ for all $x\in D$ and $\lambda\in\mathbb{N}_{0}$.
  For $m\in A_{\lambda}(x)$, $\lambda\geq 2$, we may estimate 
  \begin{align}\label{E:Kmlambdadistance}
    2(\lambda-1)M\leq |x-y|\leq2\sqrt{n}(\lambda+1) M\sjump \;\;\forall\;y\in K_m.
  \end{align}
  Hence, if $\lambda\geq 2$, $y\in K_m$ and $m\in A_\lambda(x)$, it follows that
  $$p_m(x)\leq \left(2(\lambda-1) M\right)^{2-n}\int_{\mathbb{R}^n}1\;d\mu_m(y)
  = \left(2(\lambda-1) M\right)^{2-n},$$
  which implies that $$e^{-\frac{4}{\delta p_m(x)}}\leq e^{-\frac{4}{\delta}\left(2(\lambda-1) M\right)^{n-2}}.$$
  This, together with \eqref{E:basicestp}, implies that
  \begin{align*}
     \phi(x)&\leq \card\left(A_0(x)\cup A_1(x)\right)e^{-4}+c(n)\sum_{\lambda=2}^\infty
     \card\left(A_\lambda(x) \right)e^{-\frac{4}{\delta}\left(2(\lambda-1) M\right)^{n-2}}\\
     &\leq (3^n+1)e^{-4}+c_n\sum_{\lambda=2}^{\infty}\lambda^{n-1}e^{-\frac{4}{\delta}\left(2(\lambda-1) M\right)^{n-2}}<\infty,
  \end{align*} 
  i.e., the series $\phi(x)$ is convergent. In fact, it follows from the Weierstra\ss\,$M$-test 
  that the series $\phi$ converges absolutely uniformly on $D$, thence $\phi$ is continuous on $D$. 
  It also follows that the value of $\phi(x)$ is independent of the enumeration of $\mathbb{Z}$ in the sum, 
  in particular, 
  $$\phi(x)=\sum_{m\in\mathbb{Z}^n}e^{-\frac{4}{\delta p_m(x)}}\sjump\forall x\in D.$$ Finally, 
  the above estimate shows that $\phi$ is bounded on $D$.
   
   To show that $\phi\in\mathcal{C}^\infty(D)$, it suffices to show that for any given $x\in D$,  $j\in\mathbb{N}$ and
   \begin{align}\label{D:Djdefinition}
     \mathcal{D}_j:=\frac{\partial^{j}}{\partial x_1^{j_1}\dots\partial x_n^{j_n}}
   \end{align}
   with $j_1+\dots+j_n=j$, the series
   \begin{align}\label{E:derivativephi}
     \sum_{\lambda=0}^\infty\sum_{m\in A_\lambda(x)}\mathcal{D}_j\left(e^{-\frac{4}{\delta p_m(x)}} \right) 
   \end{align}  
   is absolutely uniformly convergent on some neighborhood of  $x$. Roughly speaking, this convergence follows because
   $\mathcal{D}_j\left(e^{-\frac{4}{\delta p_m(x)}} \right)$ is the finite sum of product of two terms, one decreases like 
   (positive powers of) $e^{-\lambda}$ 
   while the other grows at most polynomially in $\lambda$ 
   when $\lambda\to\infty$. 
   To wit, let $x\in D$ be given. Let $j\in\mathbb{N}$, and $\mathcal{D}_j$ be some differential operator as in
   \eqref{D:Djdefinition}. Then, by Fa{\'a} di Bruno's formula, see, e.g.,  \cite{Hardy06}, it follows that
   \begin{align*}
     \mathcal{D}_j(e^{-\frac{4}{\delta p_m}})=\sum_{\pi\in\Pi}\left(\frac{\partial^{|\pi|}}{\partial t^{|\pi|}}
     e^{-\frac{4}{\delta t}} \right)_{|_{t=p_m}}\prod_{B\in\pi}\frac{\partial^{|B|}p_{m}}{\prod_{k\in B}\partial \xi_k},
   \end{align*} 
   where $\Pi$ is the set of all partitions of $\{1,\dots,j\}$, $|\pi|$ the cardinality of an element $\pi$ of $\Pi$,
   and 
   $$\textstyle\frac{\partial^{j_1}}{\partial x_1^{j_1}}=\frac{\partial}{\partial \xi_1}\dots\frac{\partial}{\partial \xi_{j_1}},\;
   \frac{\partial^{j_2}}{\partial x_2^{j_2}}=\frac{\partial}{\partial \xi_{j_1+1}}\dots\frac{\partial}{\partial \xi_{j_1+j_2+1}},\dots\;\;.$$
   For $\lambda\geq 2$ and $m\in A_{\lambda}(x)$, it follows from \eqref{E:Kmlambdadistance} that
   \begin{align*}
      \left(\frac{\partial^{j}}{\partial t^{j}}
     e^{-\frac{4}{\delta t}} \right)_{|_{t=p_m(x)}}\lesssim
     e^{-\frac{4}{\delta}\left(2(\lambda-1) M\right)^{n-2}}\left(2\sqrt{n}(\lambda+1) M\right)^{(n-2)2j}.
   \end{align*}
  Similarly, we may estimate for $m\in A_\lambda$ with $\lambda\geq 2$ and $B\in\pi$, that
   $$
     \left|\frac{\partial^{|B|}p_{m}}{\prod_{k\in B}\partial \xi_k}\right|\lesssim \left(2(\lambda-1) M\right)^{2-n}.
   $$
  This implies that there is a positive integer $d$, independent of $x$, so that
   \begin{align*}
     \sum_{\lambda=2}^\infty\sum_{m\in A_\lambda(x)}\left|\mathcal{D}_j\left(e^{-\frac{4}{\delta p_m(x)}} \right) \right|
     \lesssim\sum_{\lambda=2}^\infty e^{-\frac{4}{\delta}\left(2(\lambda-1) M\right)^{n-2}}\lambda^d.
   \end{align*}
   The terms for $\lambda\in\{0,1\}$ may be estimated similarly. Hence, the series in \eqref{E:derivativephi} is 
   absolutely uniformly convergent on $D$, and $\phi\in\mathcal{C}^\infty(D)$.

   It remains to be shown that there is a constant $c>0$ such that $\Delta\phi\geq c$ on $D$. Using that $p_m$ is 
   harmonic in $D$ for all $m\in\mathbb{Z}^n$, a straightforward computation
   yields
   \begin{align*}
     \Delta\left(e^{-\frac{4}{\delta p_m}}\right)=e^{-\frac{4}{\delta p_m}} \frac{8}{\delta p_m^3}\left|\nabla p_m\right|^2
     \left(\frac{2}{\delta p_m}-1 \right).
   \end{align*}
   Hence, by \eqref{E:basicestp}, it follows that
   \begin{align}\label{E:trianglephim}
      \Delta\left( e^{-\frac{4}{\delta p_m}} \right)\geq e^{-\frac{4}{\delta p_m}}8\delta^2\left|\nabla p_m\right|^2\geq 0.
   \end{align}
   In particular, $\phi$ is subharmonic on $D$. To prove that there exists a constant $c>0$ such that $\Delta\phi\geq c$ holds on 
   $D$, it now suffices to show that there exists a $c_1>0$ such that for all $x\in D$ 
   there exists an $m_x\in\mathbb{Z}^n$ with
   $$ \Delta\left( e^{-\frac{4}{\delta p_{m_x}(x)}} \right)\geq c_1.$$
   Let $x\in D$ be given. Recall that $m(x)\in\mathbb{Z}^n$ is chosen such that $x\in\mathcal{Q}(2Mm(x), 2M)$. 
   Set $m_x=m(x)-(2,0,\dots,0)$, and observe that
   $$ 2M\leq |x-y|\leq \sqrt{2^2(n-1)^2+6^2}M\;\;\forall y\in K_{m_x}$$
   holds.
   It then follows from \eqref{E:trianglephim} that
   \begin{align}\label{E:estimatesontriangle}
      \Delta\left( e^{-\frac{4}{\delta p_{m_x}}} \right)\geq c_2\left|\nabla p_{m_x}(x)\right|^2
   \end{align}
   for some $c_2>0$ independent of $x$. To estimate $|\nabla p_{m_x}(x)|$, notice that $x_1-y_1\geq 2M$, so that
   \begin{align*}
     \left|\nabla p_{m_x}(x)\right|\geq \left|\frac{\partial p_{m_x}}{\partial x_1} \right|
     =(n-2)\int_{\mathbb{R}^n}|x-y|^{-n}(x_1-y_1)\,d\mu_{m_x}(y)
     \gtrsim M^{-n+1}.
   \end{align*}
   This, together with estimate \eqref{E:estimatesontriangle}, implies that there is a positive constant $c$ 
   such that $\Delta\phi\geq c$ on $D$.
 \end{proof}


\section{Proof of  \texorpdfstring{(1)$\Rightarrow$(3)}{(1)=>(3)}}\label{S:(1)=>(3)}


  The proof of (1)$\Rightarrow$(3) is based on the one given in \cite[Section 3.2]{GLR19}. We give brief explanations 
  when the arguments are analogous, otherwise we elaborate. Moreover, some arguments in \cite[Section 3.2]{GLR19}, 
  which lack in detail, are 
  described in full here, see, e.g., \Cref{L:smoothunion}.

  \begin{proposition}\label{L:smoothK}
    Let $K\subset\mathbb{R}^n$ be a compact set such that $\mathbb{B}\cap K\neq\varnothing$. Then, for any $\epsilon>0$, there 
    exists a relatively compact set $K_\epsilon\subset \mathbb{B}$ such that
    \begin{itemize}
      \item[(i)] $K\cap\mathbb{B}\subset K_\epsilon$,
      \item[(ii)] $\ncap(K_\epsilon)\leq\ncap(K)+\epsilon$,
      \item[(iii)] $\mathbb{B}\setminus K_\epsilon$ has smooth boundary.
    \end{itemize}
  \end{proposition}

  To prove \Cref{L:smoothK}, we use the following smooth approximation of  the union of two smoothly bounded, open sets.
  \begin{lemma}\label{L:smoothunion}
    Let $\Omega_j\Subset\mathbb{R}^n$, $j\in\{1,2\}$, be smoothly bounded, open sets. 
    Assume that $S=b\Omega_1\cap b\Omega_2$ is non-empty and that the intersection is transversal. 
    Let $W$ be an open neighborhood of $S$.
    Then there exists an open, smoothly 
    bounded set $\Omega$ such that
    \begin{itemize}
      \item[(a)] $\overline{\Omega}_1\cup\overline{\Omega}_2\subset\overline{\Omega}\subset
        \overline{\Omega}_1\cup\overline{\Omega}_2\cup W$,
      \item[(b)] $b\Omega\setminus W=b(\Omega_1\cup\Omega_2)\setminus W$.
    \end{itemize}
  \end{lemma}

  \begin{proof}[Proof of \Cref{L:smoothunion}]
    Let $r_j$ be the signed Euclidean distance function for $\Omega_j$, $j\in\{1,2\}$. Then
    $\Omega_j=\{x\in\mathbb{R}^n : r_j(x)<0\}$, $r_j\in\mathcal{C}(\mathbb{R}^n)$, and 
    there exists a neighborhood $U_j$ of $b\Omega_j$, such that 
    $r_j\in\mathcal{C}^\infty(U_j)$ and $\nabla r_j(x)\neq 0$ for all $x\in b\Omega_j$. Next, note that
    $$\rho(x):=\min\{r_1(x),r_2(x)\}=\textstyle\frac{1}{2}\bigl(r_1(x)+r_2(x)-\sqrt{\left(r_1(x)-r_2(x)\right)^2} \bigr)$$
    is a continuous defining function for $\Omega_1\cup\Omega_2$ which is smooth outside the 
    set $\{x\in\mathbb{R}^n : r_1(x)=r_2(x)\}$. A modification of $\rho$ on $W$ will yield a smoothly bounded,
    open set $\Omega$ with the properties (a) and (b).
  
    For that, note first that there exists an $\epsilon>0$ such that $|r_1(x)|+|r_2(x)|\leq2\epsilon$ implies that $x\in W$. 
    Also, note that we may assume that $W\subset U_1\cap U_2$.
    Next, choose a smooth function $\chi_\epsilon:\mathbb{R}^+_0\longrightarrow[0,1]$ such 
    that $\chi_\epsilon(0)=1$, $\chi_\epsilon(t)=0$ for all $t\geq\epsilon^2$, 
    and $\chi_{\epsilon}'(t)<0$ for $t\in(0,\epsilon^2)$. We consider
    $$R(x)=\textstyle\frac{1}{2}\left(r_1(x)+r_2(x)-\sqrt{\left(r_1(x)-r_2(x)\right)^2
    +\epsilon^2\chi_\epsilon\left((r_1(x)-r_2(x)\right)^2}\right),$$
    and set $\Omega=\{x\in\mathbb{R}^n: R(x)<0\}$. 
    Then $R$ is continuous on $\mathbb{R}^n$, and $R\in\mathcal{C}^\infty(U_1\cup U_2)$. 
    Observe that $\rho(x)\geq R(x)$ for all $x\in\mathbb{R}^n$. Thus, the first inclusion of part (a) holds. Next, if 
    $x\notin\overline{\Omega}_1\cup\overline{\Omega}_2$ and $x\in\overline{\Omega}$, then both $r_1(x)$ and $r_2(x)$ are positive, and
    \begin{align*}
       r_1(x)+r_2(x)\leq \sqrt{\left(r_1(x)-r_2(x)\right)^2+\epsilon^2\chi_\epsilon\left((r_1(x)-r_2(x)\right)^2}.
    \end{align*}
    The last two facts imply that
    \begin{align}\label{E:smoothunion}
       0<4r_1(x)r_2(x)\leq \epsilon^2\chi_\epsilon\left((r_1(x)-r_2(x))^2\right)\leq \epsilon^2
    \end{align} 
    This implies that the second to last term of \eqref{E:smoothunion} is positive. Hence, $|r_1(x)-r_2(x)|\leq \epsilon$. 
    Without loss of generality, $r_2(x)\leq r_1(x)\leq r_2(x)+\epsilon$. This, combined with \eqref{E:smoothunion},
    yields
    \begin{align}\label{E:Westimate}
       0<r_2(x)\leq\textstyle\frac{\epsilon}{2},\;\;\text{and}\;\;0<r_1(x)\leq \frac{3\epsilon}{2}.
     \end{align}  
     By our choice of $\epsilon$, it follows that $x\in W$, and, hence, (a) has been proven. 
  
    To prove (b), we first assume that $x\in b(\Omega_1\cup\Omega_2)\setminus W$. It follows that
    \begin{align*}
       |r_1(x)|+|r_2(x)|>2\epsilon, \;\;\text{and}\;\;
       x\in b\Omega_1\cup\overline{\Omega_2}^c\;\;\text{or}\;\;x\in b\Omega_2\cup\overline{\Omega_1}^c.
    \end{align*}   
    Hence, without loss of generality, $r_1(x)=0$ and $r_2(x)>0$, so that $|r_1(x)-r_2(x)|> 2\epsilon$, which implies that
    $R(x)=\rho(x)=0$. Therefore, $x\in b\Omega\setminus W$. Next, assume that $x\in b\Omega\setminus W$. Since $R(x)=0$, 
    it follows from the definition of $R$, that $r_1(x)+r_2(x)\geq 0$. Since $x\in W^c$ implies that $|r_1(x)|+|r_2(x)|>2\epsilon$,  
    we either may assume that $r_1(x)=0$ and $r_2(x)>0$, or get that both $r_1(x)$ and $r_2(x)$ are positive. In the first case,  
    $x\in b\Omega_1\setminus W$, while in the latter case, we already have shown in the argument leading up to
    \eqref{E:Westimate} that this implies that $x\in W$, which is a contradiction. This concludes the proof of (b).
  
    It remains to be shown that the gradient of $R$ does not vanish on $b\Omega$. First, consider $x\in W$, and note that, 
    without loss of generality, we may assume that $\nabla r_1$ and $\nabla r_2$ are linearly 
    independent at each point in $W$. Next, observe that $\nabla R(x)$ is a linear combination of $\nabla r_1(x)$ and $\nabla r_2(x)$. 
    Because of the linear independence of the two vectors, it follows that 
    $\nabla R(x)$ can only vanish if both the coefficients of $\nabla r_1(x)$ and $\nabla r_2(x)$ are zero. 
    A straightforward computation shows that this can only happen if for $t:=r_1(x)-r_2(x)$
    $$ \pm1=2t \sqrt{t^2+\epsilon^2\chi_\epsilon(t^2)}
    \left(1+ \epsilon^2\chi_{\epsilon}'(t^2)\right) $$
    holds, which is impossible since the left hand side is non-zero.
    Second, we consider the case that $x\in b\Omega\setminus W$. Then, by part (b), $R(y)$ is either $r_1(y)$ 
    for all $y$ near $x$ or $r_2(y)$ for all $y$ near $x$. 
    Since $\nabla r_i\neq 0$ on $b\Omega_i$, $i\in\{1,2\}$, the claim follows. 
  \end{proof}

  \begin{proof}[Proof of \Cref{L:smoothK}]
    We first show that there exists an open set $G_\epsilon\Subset\mathbb{R}^n$, with smooth boundary, 
    such that both (i) and (ii) hold for $\overline{G}_\epsilon\cap\mathbb{B}$ in place of $K_\epsilon$.
    By outer regularity of the Newtonian capacity, there exists an open set 
    $U_\epsilon\Subset\mathbb{R}^n$ such that $K\subset U_\epsilon$ and
    $$\ncap(U_\epsilon)\leq\ncap(K)+\epsilon.$$ It then follows from Urysohn's lemma that
    there exists an $f\in\mathcal{C}^\infty(\mathbb{R}^n)$ such that
    $f=0$ on $K$ and $f=1$ on $U_\epsilon^c$. By Sard's lemma, the image of the critical points of $f$ is of 
    Lebesgues measure $0$. Thus, we may choose a $\tau\in (0,1)$ such that
    $\nabla f(x)\neq 0$ for all $x\in\mathbb{R}^n$ with $f(x)=\tau$. Set
    $$G_\epsilon=\{x\in\mathbb{R}:f(x)<\tau\}.$$
    Then, by construction, $G_\epsilon$ is a smoothly bounded, open set. 
    Moreover, $K\subset G_\epsilon\Subset U_\epsilon$, so that monotonicity of the Newtonian capacity yields
    $$\ncap(\overline{G}_\epsilon)\leq\ncap(U_\epsilon)\leq\ncap(K)+\epsilon.$$
    After possibly slightly decreasing the value of $\tau$, within the range of $(0,1)$, we may assume that $bG_\epsilon$ 
    and $b\mathbb{B}$ intersect transversally. Let $R>0$ be such that $G_\epsilon\Subset\mathbb{B}(0,R)$. 
    Apply
    \Cref{L:smoothunion} with $\Omega_1=G_\epsilon$ and $\Omega_2=\mathbb{B}(0,R)\setminus\overline{\mathbb{B}}$. 
    It follows from outer regularity of the Newtonian capacity, 
    that we may choose a neighborhood $W$ of $b\Omega_1\cap b\Omega_2$ in \Cref{L:smoothunion}
    such that
    $$\ncap(\overline{G}_\epsilon\cup W)\leq\ncap(\overline{G}_\epsilon)+\epsilon. $$
    Let $\Omega$ be the open, smoothly bounded set constructed in \Cref{L:smoothunion} for the triple $(\Omega_1,\Omega_2,W)$.
    Set $K_\epsilon
    =\overline{\Omega}\cap\mathbb{B}$. It then follows that $\mathbb{B}\setminus K_\epsilon$ is a smoothly 
    bounded, open set, and, by property (a) of \Cref{L:smoothunion}, 
    $K_{\epsilon}\subset\overline{G}_\epsilon\cup\overline{W}\cap\mathbb{B}$. The latter implies that
    $$\ncap(K_\epsilon)\leq\ncap(\overline{G}_\epsilon)+\epsilon\leq\ncap(K)+2\epsilon,$$
    which concludes the proof.
  \end{proof}

  We are now able to prove a continuity property for the Newtonian capacity which is crucial for our proof of
  the implication (1)$\Rightarrow$(3) of \Cref{T:MainTheorem}.

  \begin{proposition}\label{P:eigenvalueconv}
     Let $\{K_j\}_{j\in\mathbb{N}}\subset\overline{\mathbb{B}(0,1)}$ be a sequence of 
     compact sets such that $\ncap(K_j)>0$ for all $j\in\mathbb{N}$, and $\lim_{j\to\infty}\ncap(K_j)=0$. 
     Suppose each $D_j=\mathbb{B}\setminus K_j$ 
     has a $\mathcal{C}^\infty$-smooth boundary. Then $\lim_{j\to\infty}\lambda_1(D_j)=\lambda_1(\mathbb{B})$.
   \end{proposition}

  \begin{proof}[Proof of \Cref{P:eigenvalueconv}]
     Following the arguments in \cite{GLR19}, subsequent to (3.8),
     it suffices to show that there exists a sequence $\{g_j\}_{j\in\mathbb{N}}$ of functions on $\overline{D_j}$ such that
     \begin{itemize} 
       \item[(a)] $g_j$ is positive and harmonic in $D_j$,
       \item[(b)] $g_j\in\mathcal{C}(\overline{D}_j)$,
       \item[(c)]$g_j$ equals $1$ on $b D_j\cap\mathbb{B}$ and  nonnegative
         on $bD_j\cap b\overline{\mathbb{B}}$,
       \item[(d)] $\lim_{j\to\infty}g_j=0$ in $L^1(D_j)$
     \end{itemize}
     For each $j\in\mathbb{N}$, let $\nu_j$ be the equilibrium measure for the 
     compact set $K_j$. Recall that $I(\nu_j)$ denotes the energy of the equilibrium measure, see \eqref{D:energy}, and set
     \begin{align*}
      g_j(x):=\frac{1}{I(\nu_j)}\int_{\mathbb{R}^n}|x-y|^{2-n}\;d\nu_j(y).
    \end{align*}
    That is, $g_j$  equals the potential associated to $\nu_j$ up to the multiplicative factor $1/I(\nu_j)$. 
    As such, an analogon of \Cref{L:basicproppot} holds for $g_j$. To show that $g_j$ 
    is positive on $D_j$, note that $|x-y|$ is bounded from the above by $2$ for $x\in D_j$ and $y\in K_j$. 
    Hence, $$g_j\geq \frac{2^{2-n}}{I(\nu_j)}>0 \text{ on }D_j.$$
    Furthermore, by property (i) of \Cref{L:basicproppot}, $g_j$ is harmonic on $D_j$, hence (a) holds. Of course, 
    harmonicity of $g_j$ also implies that $g_j\in\mathcal{C}^\infty(D_j)$. Since the boundary of $D_j$ is smooth, in particular 
    satisfies the cone condition of Poincar{\'e} at every boundary point, $\lim_{x\to y}g_j(x)=g_j(y)=1$ for all $y\in K_j$, 
    see Theorem~4.3 and the subsequent Remark~1 in Ch. IV in \cite{Land72}. Thus, both (b) 
    and (c) hold. To prove that $g_j\to 0$ in $L^1(D_j)$ as $j\to\infty$, we compute first for $y\in\overline{\mathbb{B}}$
    \begin{align*}
       \int_{D_j}|x-y|^{2-n}\;dV(x)&\leq\int_{\mathbb{B}}|x-y|^{2-n}\;dV(x)\\
     &\leq\int_{\mathbb{B}(y;2)}|x-y|^{2-n}\;dV(x)=\int_{\mathbb{B}(0;2)}|\zeta|^{2-n}\;dV(\zeta)=2c_n,
    \end{align*}
    where $c_n$ is the surface area of the $(n-1)$-sphere. After an application of Fubini's Theorem, it follows that
    \begin{align*} 
      \int_{D_j}g_j(x)\;dV(x)=\frac{1}{I(\nu_j)}\int_{\mathbb{R}^n}\int_{D_j}|x-y|^{2-n}\;dV(x)\;d\nu_j(y)
      \leq\frac{2c_n}{I(\nu_j)}\to 0 \text{ as }j\to\infty,
    \end{align*}
    which completes the proof.
  \end{proof}

  \begin{proof}[Proof of (1)$\Rightarrow$(3) of \Cref{T:MainTheorem}]
    Suppose that the Poincar{\'e} inequality holds on $D$, i.e., $\lambda_1(D)>0$. The proof is done by contradiction, i.e.,  
    we assume that $\rho_{D}=\infty$.
    First, let us choose an $M\geq1$ such that
    $$M^2\lambda_1(D)>\lambda_1(\mathbb{B}).$$
    Second, let $\{\epsilon_j\}_{j\in\mathbb{N}}$ be a positive sequence in $\mathbb{R}$ which converges to $0$ as $j\to\infty$.
    Then, since $\rho_{D}=\infty$, for each $j\in\mathbb{N}$, there is an $x_j\in\mathbb{R}^n$ such that
    $$\ncap\bigl(\overline{\mathbb{B}(x_j;M)}\cap D^c\bigr)<\epsilon_j/2.$$
    Set
    $$
    \mathfrak{K}_j:=\bigl\{x\in\overline{\mathbb{B}}: Mx+x_j\in \overline{\mathbb{B}(x_j;M)}\cap D^c\bigr\}, \;j\in\mathbb{N}.
    $$
    It then follows from (i) and (ii) of \Cref{L:basicpropcap} that $\ncap(\mathfrak{K}_j)\leq\epsilon_j/(2M^{n-2})$ 
    for all $j\in\mathbb{N}$. Applying \Cref{L:smoothK} with $K=\mathfrak{K}_j$ 
    and 
   $\epsilon=\epsilon_j/(2M^{n-2})$ gives a relatively compact set $K_j\subset\overline{\mathbb{B}}$ such that $\ncap(K_j)\leq\epsilon_j$ 
   and $D_j:=\mathbb{B}\setminus K_j$ is smoothly bounded. 
   Note that
   \Cref{L:basicproplambda} yields
    $$ 
    \lambda_1(D_j)=M^2\lambda_1(MD_j+x_{j})\geq M^2\lambda_1(D)\;\;\sjump\forall j\in\mathbb{N}.
    $$
    Hence, by the choice of $M$, there exists an $\epsilon>0$ such that
    $$ \lambda_1(D_j)>\lambda_1(\mathbb{B})+\epsilon\;\;\sjump\forall j\in\mathbb{N},$$
    which is a contradiction to \Cref{P:eigenvalueconv}.
  \end{proof}

  The last proof also yields the sharp upper bound for $\lambda_1(D)$ in terms of the strict Newtonian capacity inradius as stated 
  in \Cref{C:bestupperbound}.


  \section{Proof of  \texorpdfstring{(1)$\Leftrightarrow$(5)}{(1)<=>(5)}}\label{S:(1)<=>(5)}


  \begin{proof}[Proof of (1)$\Leftrightarrow$(5) of \Cref{T:MainTheorem}]
    By the remark at the beginning of \Cref{S:(1)<=>(4)}, it suffices to show that the Poincar{\'e} inequality 
    \eqref{E:PD} holds on $D$ with $C>0$ if and only if
    \begin{align}\label{E:dbarstarmCR}
      \|v\|_{L^2_{0,m}(D)}\leq C\|\dbarstar_{m-1}v\|_{L^2_{0,m-1}(D)}\;\;\sjump\forall v\in\dom(\dbarstar_{m-1})
      \cap (\ker(\dbarstar_{m-1}))^\perp.
    \end{align}
    We shall use the fact the $\Omega^{0,m}_c(D)$ is dense in $\dom(\dbarstar_{m-1})$ with respect to the graph norm
    \begin{align}\label{E:graphnorm}
      \bigl(\|v\|_{L_{0,m}(D)}^2+\|\dbarstar v\|_{L^2_{0,m-1}(D)}^2\bigr)^\frac{1}{2}, 
    \end{align}
    see Proposition 2.3 in \cite{Straube10}.
    Next, one easily computes  for $u=\sum_{j=1}^m u_j\widehat{d\bar{z}^j}\in\Omega^{0,m-1}_c(D)$, with 
    $\widehat{d\bar{z}^j}=d\bar{z}^1\wedge\ldots\wedge d\bar{z}^{j-1}\wedge d\bar{z}^{j+1}\wedge\ldots\wedge d\bar{z}^m$, and
    $v=\nu d\bar{z}^1\wedge\ldots\wedge d\bar{z}^m\in\Omega^{0,m}_c(D)$ that
    \begin{align*}
       \dbar_{m-1}u=\sum_{j=1}^m(-1)^{j+1}\frac{\partial u_j}{\partial\bar{z}_j}d\bar{z}^1\wedge\ldots\wedge d\bar{z}^n, \text{ and }
       \dbarstar_{m-1}v=\sum_{j=1}^m(-1)^j\frac{\partial \nu}{\partial z_j}\widehat{d\bar{z}^j}
    \end{align*}
    holds.
    Hence, 
    $$
    \|\dbarstar v\|_{L^2_{0,m-1}(D)}^2=(\dbar_{m-1}\dbarstar_{m-1}v,v)_{L^2_{0,m}(D)}=-\frac{1}{4}(\Delta\nu,\nu)_{L^2(D)}=
    \frac{1}{4}\|\nabla\nu\|_{L^2(D)}^2 $$
    follows, which, together with the density result, implies that \eqref{E:dbarstarmCR} holds with constant $C>0$ whenever \eqref{E:PD} does. 
  
    To show 
    the reverse implication, i.e., \eqref{E:dbarstarmCR}$\Rightarrow$\eqref{E:PD},  it now suffice to 
    show that $\ker(\dbarstar_{m-1})=\{0\}$. If $v\in\ker(\dbarstar)$, 
    with $v=\nu d\bar{z}^1\wedge\ldots\wedge d\bar{z}^m$ then, by the above mentioned density result,
    there exists a sequence $\{v_k\}_{k\in\mathbb{N}}\subset\Omega^{0,m}_c(D)$ which converges to $v$ in \eqref{E:graphnorm}. 
    But this implies that 
    $\{\nu_k\}_{k\in\mathbb{N}}\subset\mathcal{C}^\infty_c(D)$ is Cauchy in  the Sobolev space $H^1_0(D)$, 
    with $\nabla\nu_k\to 0$ and $\nu_k\to \nu$ as $k\to\infty$. Since $H_0^1(D)$ cannot contain any 
    nontrivial functions which are constant on the components of $D$, it follows that $v=0$ in $L^2_{0,m}(D)$, see 
    Lemma 2.11 and the proof of Proposition 2.9 in \cite{GLR19} for further details.
  \end{proof}

\section*{Statements and Declarations}
The author has no conflict of interest to declare.
\bibliographystyle{acm}
\bibliography{References}

\end{document}